\documentclass[11pt]{article}
\usepackage{amssymb,amsmath,latexsym,amsthm,amsfonts}
\usepackage{mathrsfs}
\usepackage{subfigure}
\usepackage{tikz}
\usepackage[textwidth=2cm]{todonotes}
\usetikzlibrary{patterns}
\usepackage{hyperref,enumerate}
\usepackage[affil-it]{authblk}
\usepackage{cleveref}

\presetkeys{todonotes}{fancyline, size=\scriptsize}{}

\newtheorem{thm}{Theorem}[section]
\newtheorem{cor}[thm]{Corollary}
\newtheorem{lem}[thm]{Lemma}
\newtheorem{lemma}[thm]{Lemma}
\newtheorem{prop}[thm]{Proposition}
\newtheorem{conj}[thm]{Conjecture}
\newtheorem{obs}[thm]{Lemma}

\newtheorem{problem}[thm]{Problem}

\DeclareMathOperator{\tc}{tc}
\DeclareMathOperator{\ram}{r}
\DeclareMathOperator{\pp}{pp}
\DeclareMathOperator{\cp}{cp}

\newcommand{\C}{\mathcal{C} }

\date{\today}
\author{Sebasti\'an Bustamante and Maya Stein
}
\affil{Department of Mathematical Engineering\\ University of Chile}

\title{\bf Monochromatic tree covers and Ramsey\\ numbers for set-coloured graphs}

\begin{document}
\maketitle
\thispagestyle{empty}

\begin{abstract}
 We consider a generalisation of the classical Ramsey theory setting to a setting where each of the edges of the underlying host graph is coloured with a {\em set} of colours (instead of just one colour). 
 We give bounds for monochromatic tree covers in this setting, both for an underlying complete graph, and an underlying complete bipartite graph.
 We also discuss a generalisation of Ramsey numbers to our setting and propose some other new directions.
 
  Our results for tree covers in complete graphs imply that  a stronger version of Ryser's conjecture holds for $k$-intersecting $r$-partite $r$-uniform hypergraphs: they have a transversal of size at most $r-k$. (Similar results have been obtained by Kir\'aly et al., see below.) However, we also show that the bound  $r-k$  is not best possible in general.
\end{abstract}

\section{Introduction}
\subsection{Set-colourings}
We consider complete (and complete bipartite) graphs $G$ whose edges are each coloured with a set of $k$ colours, chosen among $r$ colours in total. 
That is, we consider functions $\varphi:E(G)\to {[r]\choose k}$, where ${[r]\choose k}$ is the set of $k$-element subsets of $[r] := \{1, 2, \dots, r\}$.
We call any such $\varphi$ an  $(r,k)$-colouring (so, the usually considered $r$-colourings for Ramsey problems are $(r,1)$-colourings). 
Colourings of this type, and related concepts, appeared in~\cite{XSSL09}, and in~\cite{BGS11,BT79,Hed09}, respectively. 
We consider Ramsey-type problems for $(r,k)$-coloured host graphs.

\subsection{Tree covers in complete graphs}
The first problem we consider is the tree covering problem. 
In the traditional setting~\cite{EGP91, Gya77, HK96}, one is interested in the minimum number $\tc_r(K_n)$ such that each $r$-colouring of $E(K_n)$ admits a cover with $\tc_r(K_n)$ monochromatic trees (not necessarily of the same colour). 
The following conjecture has been put forward by Gy\'arf\'as:

\begin{conj}[Gy\'arf\'as \cite{Gya77}]
 \label{conj:Gya77}
 For all $n\geq 1$, we have $\tc_r(K_n) \leq r-1.$
\end{conj}

Note that this conjecture becomes trivial if we replace $r-1$ with $r$, as for any colouring, all monochromatic stars centered at any fixed vertex cover~$K_n$. 
Also, the conjecture is tight when $r-1$ is a prime power, as we will show in Section~\ref{sec:easy}. 
Conjecture~\ref{conj:Gya77} holds for $r \leq 5$, due to results from Duchet~\cite{Duc79} and Gy\'arf\'as~\cite{Gya77}, through a connection to Ryser's conjecture. 
We shall discuss this connection at the end of the introduction.

In our setting, for a given graph $G$ we define the {\em tree cover number} $\tc_{r,k}(G)$ as the minimum number $m$ such that each $(r,k)$-colouring of $E(G)$ admits a cover with $m$ monochromatic trees. 
In this context, a monochromatic tree in $G$ is a tree $T\subseteq G$ such that there is a colour~$i$ which, for each $e\in E(T)$, belongs to the set of colours assigned to $e$.

Note that deleting $k-1$ fixed colours from all edges, and, if necessary, deleting some more colours from some of the edges, we can produce an $(r-k+1)$-colouring from any given $(r,k)$-colouring. 
So, Conjecture~\ref{conj:Gya77}, if true, implies that $\tc_{r,k}(K_n)\leq r-k$.

\begin{conj}
 \label{conj:comp_r-k}
 For all $n\geq 1$ and $r > k \geq 1$, we have $\tc_{r,k}(K_n) \leq r-k.$
\end{conj}

Clearly, the bound from Conjecture~\ref{conj:comp_r-k} is tight for $k = r-1$, and it is also tight for  $k = r-2$, as a consequence of Lemma~\ref{frr-2} (see Section~\ref{sec:easy}). 
In~\cite{Kir11}, Kir\'aly proved this bound for $k > r/2$. 
 Lemmas~\ref{compr-k} and \ref{compr-k2} confirm the bound from Conjecture~\ref{conj:comp_r-k}  for $k \geq r/2 - 1$.
After the original version of the present paper was submitted, Kir\'aly and T\'othm\'er\'esz~\cite{KT17} confirmed the bound for $k > r/4$.

But in general, the bound $ r-k$ is not tight.
The smallest example (in terms of $r$ and $k$) corresponds to  $r=5$ and $k=2$, and will be discussed in Section~\ref{sec:comp_gr}.

\begin{thm}\label{thm:52}
 For all $n\geq 4$, we have $\tc_{5,2}(K_n) = 2$.
\end{thm}

\subsection{Tree covers in complete bipartite graphs}
Tree coverings have also been studied for complete bipartite graphs $K_{n,m}$. 
Chen, Fujita, Gy\'arf\'as, Lehel and T\'oth~\cite{CFGLT12} proposed the following conjecture.

\begin{conj}{$\!\!$\rm\bf\cite{CFGLT12}}
 \label{conj:CFGLT12}
 If $r > 1$ then $\tc_{r,1}(K_{n,m}) \leq 2r-2$, for all $n, m\geq 1$.
\end{conj}

Notice that Conjecture~\ref{conj:CFGLT12} is equivalent to the same statement with $n=m$, since adding copies of some vertex in the smaller part does not modify the tree cover number.
It is shown in~\cite{CFGLT12} that Conjecture~\ref{conj:CFGLT12} is tight; that it is true for $r \leq 5$; and that $\tc_{r,1}(K_{n,m}) \leq 2r-1$ for all $r,n,m\geq 1$.
Thus, in our setting, we can use the argument from above, deleting $k-1$ fixed colours, to see that $\tc_{r,k}(K_{n,m}) \leq 2r-2k+1$ (see Section~\ref{sec:easy} for details). 
But we can do better than this:

\begin{thm}
 \label{thm:mainBip}
 For all $r,k,n,m$,
 \begin{equation*}
  \tc_{r,k}(K_{n,m}) \leq
  \begin{cases}
   r-k+1,   & \text{if}\ k \geq r/2         \\
   2r-3k+1, & \text{if}\ r/2 > k \geq 2r/5 \\
   2r-3k+2, & \text{otherwise.}
  \end{cases}
 \end{equation*}
\end{thm}

For the case $k \geq r/2$, our bound is sharp for large graphs:

\begin{thm}\label{thm:lowerbound}
 For each $r,k$ with $r > k$ there is $m_0$ such that if $n\geq m\geq m_0$ then $\tc_{r,k}(K_{n,m}) \geq \max\{r-k+1,r-k+ \lfloor \frac rk \rfloor -1\}$.
\end{thm}

Theorems~\ref{thm:mainBip} and~\ref{thm:lowerbound} will be proved in Sections~\ref{sec:upperbip} and~\ref{sec:lowerbip}, respectively.

%

\subsection{Set-Ramsey numbers}

Classical Ramsey problems naturally extend to $(r,k)$-colourings. 
Define the {\it set-Ramsey number} $\ram_{r,k}(H)$ of a graph $H$ as the smallest~$n$ such that every $(r,k)$-colouring of $K_n$ contains a monochromatic copy of~$H$. 
(As above, a monochromatic subgraph $H$ of $G$ is a subgraph $H\subseteq G$ such that  there is a colour~$i$ that appears on  each $e\in E(H)$.) 
So the usual $r$-colour Ramsey number of $H$ equals $\ram_{r,1}(H)$. Note that $\ram_{r,k}(H)$  is  increasing in $r$ if $H$ and~$k$ are fixed, and decreasing in $k$ if $H$ and $r$ are fixed.

There is a connection between the set-Ramsey number $\ram_{r,k}(H)$ and another  Ramsey-type concept, which was introduced by Erd\H os, Hajnal and Rado in~\cite{EHR65}. 
Let $f_r(H)$ be the  smallest number $n$ such that every $r$-colouring of the edges of~$K_{n}$ contains a copy of $H$ whose edges use at most $r-1$ colours. 
Note that each $(r,r-1)$-colouring $\varphi$ of $K_n$ corresponds to an $r$-colouring $\varphi'$ of~$K_n$, by giving each edge the colour it does {\em not} have in $\varphi$. 
Moreover, observe that $\varphi$  contains a monochromatic copy of $H$  if and only if  $\varphi'$ contains a copy of $H$ that uses at most $r-1$ colours. So $\ram_{r,r-1}(H) = f_r(H)$.

Alon, Erd\H os, Gunderson and Molloy~\cite{AEGM02} study the asymptotic behaviour of $f_{r}(K_n)$. 
See also~\cite{HM99} for related results. 
Chung and Liu~\cite{CL78}, and  Xu et al.~\cite{XSSL09}, study $f_{3}(K_t)=\ram_{3,2}(K_t)$ for small $n$.

We determine $\ram_{4,2}(K_3)$ in Corollary~\ref{42K3}. This makes use of a lower bounds for $\ram_{r,k}(C_\ell)$ for cycles $C_\ell$ of odd length $\ell$ given in Theorem~\ref{r_lowerboundC} (another bound is given in Proposition~\ref{p_dj16}). 

 We also determine for which values of $r,k,t$ we have  $\ram_{r,k}(K_t)=t$ and give upper bounds for $\ram_{r,k}(K_t)$ using Tur\'an's theorem (see Proposition~\ref{boundTuran_f} and the discussion before the proposition). 
All of these results can be found in Section~\ref{sec:ramsey}.

\subsection{Other directions}
In Section~\ref{sec:conc_rem}, we summarize all open problems regarding the topics discuseed so far (tree covers in complete and complete bipartite graphs for set-colourings, and set-Ramsey numbers for complete graphs and for cycles). Furthermore, we propose several new directions that could be studied for set-colourings. Those are tree partition problems, path partition problems, and cycle partition problems. We also include some basic observations. In particular, and perhaps unexpectedly, the cycle partition number for $(3,2)$-coloured complete graphs turns out to be $2$ (see Section~\ref{cycles} for a definition of this number).

\subsection{Tree covers and  Ryser's conjecture}

Finally, let us explain the connection between Conjecture~\ref{conj:Gya77} and Ryser's conjecture~\cite{Hen70}. 
The latter conjecture states that $\tau(\mathcal{H}) \leq (r-1)\nu(\mathcal{H})$ for each $r$-partite $r$-uniform hypergraph $\mathcal{H}$ with $r>1$, where $\tau(\mathcal{H})$ is the size of a smallest transversal (vertex set intersecting every edge) of $\mathcal{H},$ and  $\nu(\mathcal{H})$ is the size of a largest matching in $\mathcal{H}$.

Now, each $r$-partite $r$-uniform hypergraph $\mathcal{H}$ gives rise to a graph~$G$ on vertex set $E(\mathcal{H})$, whose edges are coloured with subsets of colours in~$[r]$: If hyperedges $v,w$ of $\mathcal H$ intersect, say in partition classes $i_1,\ldots , i_\ell$, then the edge~$vw$ of $G$ carries all colours $i_1,\ldots , i_\ell$, and if hyperedges $v,w$ do not intersect, then $vw$ is not an edge of $G$. 
Note that all monochromatic components of $G$ are complete. 
Moreover, this is a $1$-to-$1$ correspondence, as we can also construct from any graph~$G$ coloured in this way a unique (up to isomorphism) $r$-partite $r$-uniform hypergraph $\mathcal{H}$. 
It is easy to observe that $\tau(\mathcal{H})$ equals the minimum number of monochromatic trees covering~$V(G)$.

Because of this correspondence, Conjecture \ref{conj:Gya77} is equivalent to Ryser's conjecture for {\em intersecting} hypergraphs (those with $\nu(\mathcal{H}) = 1$). 
Namely, for these hypergraphs, every two hyperedges intersect, and thus $G$ is complete.
From the given set-colouring, we can get to an $r$-colouring by simply deleting colours on some of the edges (note that it does not matter if this disconnects some of the monochromatic components), and thus, Conjecture \ref{conj:Gya77} implies Ryser's conjecture for  intersecting hypergraphs.
For the other direction, given an $r$-colouring  of $K_n$, we can add colours on some of the edges, making the monochromatic components complete. 
Note that this does not affect the sizes of the monochromatic components.
So, Ryser's conjecture for intersecting hypergraphs implies  Conjecture \ref{conj:Gya77}.


\section{$r$-colourings and $(r,k)$-colourings}\label{sec:easy}

This section contains several easy bounds on tree cover numbers for $(r,k)$-colourings, often in terms of bounds on tree cover numbers for $r$-colourings.
We start with the trick mentioned in the introduction.

\begin{obs}
 \label{obsred}
 For every graph $G$, if there exists $f(r)$ such that $\tc_{r,1}(G) \leq f(r)$, then  $\tc_{r,k}(G) \leq f(r-k+1).$
\end{obs}
\begin{proof}
 Given any $(r,k)$-colouring $\varphi$ of $G,$ we can construct an edge-colouring $\varphi'$ of $G$ by arbitrarily fixing $k-1$ colours, deleting them from every edge of $G$, and, if necessary, deleting some more colours from the edges until we are left with a $(r-k+1)$-colouring. 
 Each monochromatic component of  $\varphi'$ is contained in a monochromatic component of  $\varphi$.
\end{proof}

So, the trivial upper bound $\tc_{r,1}(K_n)\leq r$ implies that  $\tc_{r,k}(G) \leq r-k+1$, and this bound drops to $r-k$ if Conjecture \ref{conj:Gya77} is true.
Similarly, Conjecture~\ref{conj:CFGLT12}, if true, or the above mentioned bound of $2r-1$ from~\cite{CFGLT12}, combined with Lemma~\ref{obsred}, yield bounds for $ \tc_{r,k}(K_{n,m})$, which, however, are improved by our Theorem~\ref{thm:mainBip}.

For the following lemma, notice that in an $(r,k)$-coloured graph, every set  of  $r-k+1$ colours from $[r]$ contains at least one colour from each edge.

\begin{obs}\label{star}
 Let $\varphi$  be an $(r,k)$-colouring of $K_{1,n}$, and let $\mathcal{C} \subseteq [r]$ with $|\mathcal{C}| =r-k+1$. 
 Then we can cover the vertices of $K_{1,n}$ by $r-k+1$ monochromatic stars, each using a different colour from $\mathcal{C}$.
\end{obs}

An easy lower bound on $\tc_{r,k}(G)$ can be obtained by splitting colours.

\begin{obs}
 \label{obsred3}
 If there exists a function $f(r)$ such that $\tc_{r,1}(G) \geq f(r)$, then $\tc_{r,k}(G) \geq f(\lfloor r/k\rfloor)$.
\end{obs}

\begin{proof}
 It is enough to observe that, with no effect on the number of monochromatic components needed to cover $G$, we can modify any $r$-colouring  of~$G$  to an $(rk,k)$-colouring  by replacing each colour with  a set of $k$ new colours.
\end{proof}

Let us now see how a given $(r,k)$-colouring of a graph can be extended to a larger graph, without affecting the tree cover number.
To this end, for an $(r,k)$-colouring $\varphi$ of a graph $G$ we define $\tc(G,\varphi)$ as the minimum number of monochromatic trees induced by $\varphi$ needed to cover the vertices of $G$.

\begin{obs}
 \label{obsred4}
 Let  $\varphi$ be an $(r,k)$-colouring of $K_{n,m}$. 
 Then for all $n'\geq n$, $m'\geq m$ there is an $(r,k)$-colouring $\varphi'$ of $K_{n',m'}$ such that  $\tc(K_{n',m'},\varphi') =  \tc(K_{n,m},\varphi)$.
\end{obs}

\begin{proof}
 Duplicate any vertex $x$, together with its incident edges and their colours, to obtain an $(r,k)$-colouring of $K_{n+1,m}$ (or of $K_{n,m+1}$).
 Since all monochromatic components have stayed the same, modulo a possible duplication of $x$, the new colouring of  $K_{n+1,m}$ (or of $K_{n,m+1}$) cannot be covered with fewer than $\tc(K_{n,m},\varphi)$ monochromatic trees. 
 By applying induction, we are done.
\end{proof}

In the same way, we obtain the analogous statement for the complete graph (the edge between the two copies of $x$ can receive any set of colours).

\begin{obs}
 \label{obsred5}
 Let  $\varphi$ be an $(r,k)$-colouring of  $K_{n}$. 
 Then for each $n'\geq n$ there is an $(r,k)$-colouring $\varphi'$ of $K_{n'}$ such that  $\tc(K_{n'},\varphi') = \tc(K_n,\varphi)$.
\end{obs}

It is well known that Ryser's conjecture, if true, is tight for infinitely many values of $r$. 
Namely\footnote{The construction is as follows. 
 Consider the complete graph $K_n$ on the point set of an affine plane of order $r-1$ (with $r-1$ a prime power). 
 Colour $uv$ with colour $i$ if  the $i$th partition is the unique partition $P_{u,v}$ which has a block covering both $u$ and $v$. 
 As each monochromatic component of $K_n$ corresponds to a block of the affine plane, and thus has $r-1=n/(r-1)$ vertices, we need at least $r-1$ monochromatic trees to cover  $V(K_n).$}, 
if $r-1$ is a prime power, then  $K_{(r-1)^2}$ has an $r$-colouring $\varphi$ with $\tc(K_{(r-1)^2},\varphi) \geq r - 1$. 
So, using Lemma~\ref{obsred5} and Lemma \ref{obsred3} we get:

\begin{obs}\label{loboco}
 Let  $r\geq k$ with $r-1$ a prime power, and let $n \geq (r-1)^2$. 
 Then there is an $(r,k)$-colouring $\varphi$ of $K_n$ with $\tc(K_n,\varphi) \geq \lfloor r/k \rfloor - 1$.
\end{obs}

We close this section with another consequence of Lemma~\ref{obsred5}.

\begin{obs}
 \label{frr-2}
 For every $r \geq 3$ and $n \geq r$, we have that $\tc_{r,r-2}(K_n) \geq 2.$
\end{obs}

\begin{proof}
 Define an $(r,r-2)$-colouring of $K_r$ (on vertices $v_1,\dots, v_r$) by  assigning $v_iv_j$ colours $[r]\setminus \{i,j\}$. 
 Then no colour is connected. By Lemma~\ref{obsred5}, we are done.
\end{proof}

\section{Upper bounds for complete bipartite graphs}\label{sec:upperbip}

In this section we prove Theorem \ref{thm:mainBip}. We split the proof into two parts, covered by the following two lemmas.

\begin{lem}
 \label{lem:mt1}
 For all $n,m$, we have that
 $\tc_{r,k}(K_{n,m}) \leq
 r-k+1$, if  $k \geq r/2$, and
 $\tc_{r,k}(K_{n,m}) \leq 2r-3k+2$ otherwise.
\end{lem}

\begin{lem}
 \label{lem:mt_2k+p}
 If $r/2 > k \geq 2r/5$, then
 $\tc_{r,k}(K_{n,m}) \leq
 2r-3k+1$ for all $n,m$.
\end{lem}

We first prove the easier Lemma \ref{lem:mt1}.

\begin{proof}[Proof of Lemma \ref{lem:mt1}]
 Let $\varphi$ be an $(r,k)$-colouring of $K_{n,m}$ and fix an edge $vw \in E(K_{n,m})$. 
 By Lemma~\ref{star}, we can cover $K_{n,m}$ using $r-k+1$ stars centered at $v$ and  $r-k+1$ stars centered at $w$. 
 Since we can choose the same $r-k+1$ colours for both sets of stars, the edge $vw$, which has $k$ colours, connects $\min\{k, r-k+1\}$ of the stars.

 This bounds $\tc_{r,k}(K_{n,m})$ by $2r-3k+2$ in the case that $k\leq r-k+1$, and if $r < 2k-1$, we get a bound of $r-k+1$. 
 
 For the case $r=2k$, let $U$ be the set of vertices not covered by the $k$ components in colours $\varphi(vw)$ containing the edge $vw$. 
 Since $\varphi(ux) = [r]\setminus \varphi(vw)$ for every $u \in U$ and $x \in \{u,v\}$, we can cover the vertices of $U$ with at most two stars $S_v, S_w$ centered at $v$ and $w$, respectively, by using any colour in $[r]\setminus \varphi(vw)$. 
 If there is an edge in the complete bipartite graph induced by $U$ with a colour $c \in [r]\setminus \varphi(vw)$, then $S_v$ and $S_w$ are connected and we can cover all the vertices with $k+1 = r-k+1$ monochromatic trees. 
 If not, then every edge induced by $U$ is coloured by $\varphi(vw)$ so we can choose any of these colours to cover $U$ with just one monochromatic component, obtaining $k+1 = r-k+1$ monochromatic components covering the vertices of $K_{n,m}$ as well.
\end{proof}

We now turn to the less straightforward proof of Lemma \ref{lem:mt_2k+p}.
We need two preliminary lemmas.

\begin{lem}
 \label{edge}
 Suppose $ k<r/2 $ and let an $(r,k)$-colouring  of $K_{n,m}$ be given. 
 If there is a vertex $v$ and a set~$\mathcal C$ of $k$ colours such that no edge incident with $v$ has exactly the colours of $\C$, then there is a set~$\mathcal C'$ of $k$ colours such that
 \begin{enumerate}[(a)]
  \item no edge incident with $v$ has exactly the colours of $\C'$, and
  \item there is an edge incident with $v$ that has no colour of $\C'$.
 \end{enumerate}
\end{lem}

\begin{proof}
 Let  $v$ be as in the lemma. 
 Let $E_v$ be the set of edges incident with $v$. 
 Assume there is no set $\C'$ as required for the lemma. 
 Then, we use induction to prove that for all $i=0,1,\ldots k$ it holds that no edge in $E_v$ has exactly $i$ colours from $\C$.

 Note that the base case $i=0$ of our induction follows from the assumption that $\C$ is not the desired set $\C'$. 
 So assume the assertion holds for $i-1$, our aim is to show that it also holds for $i$. 
 If the assertion does not hold for $i$, then there is an edge $e_i$ that has exactly $i$ colours from $\C$. 
 Let $\C_i$ be the set of all colours not on $e_i$. 
 Let $\C_i'$ be a $k$-subset of $\C_i$ that has exactly $i$ elements from $[r]\setminus\C$ (such a subset exists, since $r\geq2k$ and $i\leq k$). 
 Since we assume that $\C_i'$ is not the desired set $\C'$, it follows that  there is an edge $e_i'$ in $E_v$ that has exactly the colours in $\C_i'$.

 Let $\C_i''$ be a $k$-set of  colours not on $e_i'$ such that $\C_i''$ has exactly $i-1$ elements from $\C$ (such a subset exists, since $r>2k$ and $i\leq k$). 
 Since we assume that $\C_i''$ is not the desired set $\C'$, it follows that  there is an edge  in $E_v$ that has exactly the colours in $\C_i''$. 
 But such an edge cannot exist, since we assume the inductive assertion to hold for $i-1$. 
 This finishes the inductive proof.

 Now, observe that since $\C$ is not as desired, no edge in $E_v$ has colours that form a subset of $[r]\setminus\C$. 
 Moreover, as we showed above, no edge in $E_v$ has $k$ or fewer colours from $\C$. 
 This implies that $E_v$ has no edges at all, a contradiction.
\end{proof}

\begin{lem}\label{tuplas}
 Suppose $k < r/2 $  and let an $(r,k)$-colouring $\varphi$ of $K_{n,m}$ be given. 
 If there is a vertex $v$ and a set~$\mathcal C$ of $k$ colours such that no edge incident with $v$ has exactly the colours of $\mathcal C$, then $\tc(K_{n,m},\varphi) \leq 2r-3k+1$.
\end{lem}

\begin{proof}
 Apply Lemma~\ref{edge}, for simplicity, let us call the obtained set $\C'$ still $\C$. 
 Let $vw$ be the edge given by Lemma~\ref{edge}~(b). 
 We now  proceed similarly to the proof of Lemma \ref{lem:mt1}, the only difference being that now we only take $r-k$ stars at vertex $v$ (instead of taking $r-k+1$ as in the proof of Lemma \ref{lem:mt1}). 
 The colours we choose for the stars at $v$ are exactly the colours not in~$\C$. 
 For the stars at $w$, we choose the same colours, plus one more colour, arbitrarily chosen from $\C$. 
 Note that since $vw$ has no colours from $\C$, it can be used to connect $k$ pairs of stars. 
 Hence we obtain a cover with $2r-3k+1$ monochromatic trees, as desired.
\end{proof}

We are now ready to prove  Lemma \ref{lem:mt_2k+p}.

\begin{proof}[Proof of Lemma \ref{lem:mt_2k+p}]
 Let $A,B$ be the bipartition classes of $K_{n,m}$, and fix an edge $vw \in E(K_{n,m})$ with $v\in A$ and $w\in B$. 
 Let $C_0$ be the set of vertices covered by the union of the $k$ monochromatic components that contain the edge $vw.$

 If $A':=A \setminus C_0$ is empty, then consider the star with center $v$ and leaves $B\setminus C_0$, with its inherited $(r-k,k)$-colouring. 
 By Lemma~\ref{star}, this star can be covered with at most $r-2k+1$ monochromatic stars. 
 Thus, we can cover all of $K_{n,m}$ using $k + (r-2k+1)=r-k+1 < 2r-3k+ 1$ monochromatic components in total. 
 So assume $A' \neq \emptyset$, and by symmetry, also $B':=B \setminus C_0 \neq \emptyset$.

 We claim that there is an edge $v'w'$ with $v'\in A',w'\in B'$ such that
 \begin{equation}\label{a'b'}
  \text{$|\varphi(v'w')\setminus \varphi(vw)|\geq 2(r-2k)$.}
 \end{equation}
 For the proof of~\eqref{a'b'}, start by choosing any vertex $v'\in A'$. 
 Observe that by Lemma~\ref{tuplas}, $v'$ is incident with an edge $v'x$ that has exactly colours $\varphi(vw)$. 
 Since $v'\notin C_0$, we know that $x\notin C_0$, and thus $x\in B'$. 
 Take a subset of $2(r-2k)$ colours of $\varphi(vw)$ (note that $2(r-2k)\leq k$ since $r\leq 5k/2$), and consider the corresponding monochromatic components that contain $v'x$. 
 If these components cover all of $A'\cup B'$, then we have found the desired cover of size $k+2(r-2k)< 2r-3k+1$. So assuming the contrary, there is a vertex $w'\in B'$ not covered by these components. 
 Then $v'w'$ avoids the $2(r-2k)$ colours of $\varphi(vw)$ we chose above. 
 Hence, $v'w'$ has $2(r-2k)$ colours that are not from $\varphi(vw)$, which is as desired for~\eqref{a'b'}.

 \smallskip

 So, let $v'w'$ be as in~\eqref{a'b'}, choose a set $\mathcal C_{v'w'}$ of  $2(r-2k)$ colours from $\varphi(v'w')$, and let $C_1$ be the set of vertices covered by the union of the $2(r-2k)$ monochromatic components in these colours that contain the edge $v'w'$. 
 Let $\bar C = C_0 \cup C_1.$  
 Since $k + 2(r-2k) = 2r-3k,$ we can assume that $(A \cup B) \setminus \bar C$ is non empty.

 By symmetry, assume $A'':=A  \setminus \bar C\neq\emptyset$, and let $v'' \in A''$. 
 Since $v''\notin C_1$, each colour from $\mathcal C_{v'w'}$ can appear on at most one of the edges $v''w$, $v'w$. 
 Moreover, since $v',v''\notin C_0$, no colour from $\varphi(vw)$ appears on the edges $v''w$, $v'w$. 
 So, as each of the edges $v''w$, $v'w$ has $k$ colours, there are at least $2k-2(r-2k)=6k-2r$ appearances of some colour of $\mathcal C:=[r]\setminus(\varphi(vw)\cup \mathcal C_{v'w'})$ on the edges $v''w$, $v'w$. 
 As $|\mathcal C|=3k-r$, all colours of $\mathcal C$ have to appear on both edges $v''w$, $v'w$.

 In particular, we obtain that all of $A''$ can be covered with a single star. 
 Hence we may from now on assume that also $B'':=B \setminus \bar C\neq\emptyset$ (as otherwise we are done). 
 Note that by a symmetric argument to the one given above, also for each $w'' \in B''$ all colours of $\mathcal C$ appear on both edges $vw''$, $vw'$.

 Noting that $v''$ was chosen arbitrarily in $A ''$, we can resume our observations as follows. 
 For each $v'' \in A''$, and each $w'' \in B''$,
 \begin{equation}\label{a''b''}
  \text{all colours of $\mathcal C$ appear on each of the edges $v''w$, $v'w$, $vw''$, $vw'$.}
 \end{equation}

 If there is an edge between  $A''$ and $B''$ that has one of the colours from $\mathcal C$, then, by~\eqref{a''b''}, we can cover all of $K_{n,m}$ with $k+2(r-2k)+1=2r-3k+1$ monochromatic components, and are done. So we may assume that
 \begin{equation}\label{no5}
  \text{no colour of $\mathcal C$ appears on an edge between $A''$ and $B''$.}
 \end{equation}

 Similarly, if there is an edge $e$ between $A''$ and $w'$ that has some colour $i\in\mathcal C$, we can find the desired cover (as then $e$, together with the edge $vw'$, connects the two stars in colour $i$ that cover $\{v\}\cup B''$ and $\{w\}\cup A''$). 
 We can repeat this argument for edges between $v'$ and $B''$. 
 Therefore, and as by definition $A''$ and $ B''$ avoid $C_1$, we may assume that
 \begin{equation}\label{all12}
  \text{all edges from $w'$ to  $A''$ and from $v'$ to $B''$ have colours $\varphi(vw)$.}
 \end{equation}

 So, if there are vertices $v'' \in A''$ and  $w'' \in B''$ such that $\varphi(v''w'') \cap \varphi(vw) \neq \emptyset$, then we can connect the two stars given by~\eqref{all12} using the edge $v''w''$, and obtain the desired cover. 
 Thus,
 \begin{equation}\label{no12}
  \text{no colour of $\varphi(vw)$ appears on an edge between $A''$ and $B''$.}
 \end{equation}

 Finally, putting~\eqref{no5} and~\eqref{no12} together, we see that  all edges between $A''$ and $B''$ must have colours from $\mathcal C_{v'w'}$. 
 Since we assume that $r\leq 5k/2$, this means that in fact, all of the $2(r-2k)\leq k$ colours from $\mathcal C_{v'w'}$ appear on each edge between $A''$ and $B''$. 
 Thus we can easily cover all of $A''\cup B''$ with one more monochromatic tree, and are done.
\end{proof}

\section{Lower bounds for complete bipartite graphs}\label{sec:lowerbip}

This section is devoted to the proof of Theorem~\ref{thm:lowerbound}. 
The theorem follows directly from Lemmas~\ref{cla2.1} and~\ref{lem:lowerbound} below, combined with Lemma~\ref{obsred4}.

\begin{lem}
 \label{cla2.1}
 For every $r,k,n, m$ with $r > k$ and  $n = \binom{r}{k},$ there exists an $(r,k)$-colouring $\varphi$ of $K_{n,m}$  such that $\tc(K_{n,m},\varphi) \geq r-k+1.$
\end{lem}

\begin{proof}
 Consider the complete bipartite graph with vertex sets $A = \binom{[r]}{k}$ and any set $B$. 
 Assign each edge $uv,$ with $u \in A$ and $v\in B$, the $k$-set of colours $u.$ 
 Then no set of $r-k$ or fewer monochromatic connected components in colours $i_1,\dots,i_l,\ l \leq r-k,$ can cover the vertices $a \in A$ which are subset of $[r]\setminus \{i_1,\dots,i_l\}.$
\end{proof}

The proof of the second bound is a bit more involved, using a similar technique as in~\cite{CFGLT12} by Chen, Fujita, Gy\'arf\'as, Lehel and T\'oth.

\begin{lemma}\label{lem:lowerbound}
 For every $r,k$ with $r > k$ and for $n \geq {r\choose k}\cdot{r-k \choose k}\cdot{r-2k \choose k}\cdot\ldots\cdot{r-\lfloor r/k\rfloor k\choose k}$, $m \geq  \lfloor r/k \rfloor -1$, there exists an $(r,k)$-colouring $\varphi$ of $K_{n,m}$ such that $\tc(K_{n,m},\varphi) \geq r-k+ \lfloor r/k \rfloor -1$.
\end{lemma}

\begin{proof}
 By Lemma \ref{obsred4} it suffices to prove the case for $n = {r\choose k}\cdot{r-k \choose k}\cdot{r-2k \choose k}\cdot\ldots\cdot{r-\lfloor r/k\rfloor k\choose k}$ and $m =  \lfloor r/k \rfloor -1$. 
 Let $A = [m]$ and 
 $$B = \{ x \in \binom{[r]}{k} ^m: x_i \cap x_j = \emptyset \mbox{ if } i \neq j   \}.$$ 
 We define an $(r,k)$-edge-colouring $\varphi$ of the complete bipartite graph on vertices $A \cup B$ as follows: for $i \in A$ and $x \in B,$ set $\varphi(ix) = x_i.$

 It is easy to see that every monochromatic connected component can be viewed as a star centered at some vertex in $A$. 
 Hence, in order to prove Lemma~\ref{lem:lowerbound}, all we need to show is that any set $\mathcal{S}$ of stars with their centers in $A$ that cover $A \cup B$ has cardinality at least  $r - k + m = r - k + \lfloor r/k \rfloor - 1$.

 So fix such a set $\mathcal S$. 
 For $i \in A,$ let $a_i$ be the number of stars of $\mathcal{S}$ centered at $i$. 
 Observe that we may  assume
 \begin{equation}\label{1a1}
  1\leq a_1 \leq \dots \leq a_m.
 \end{equation}

 We claim that there is a vertex $i \in A$ such that
 \begin{equation}\label{ai}
  a_i \geq r - k(m-i+1) + 1.
 \end{equation}
 Indeed, otherwise, we have $a_m \leq r-k$, $a_{m-1} \leq r-2k, \dots ,a_1 \leq r - mk$. 
 This means that we can choose a set $\mathcal{C}_m$ of $k$ colours such that no star from $\mathcal S$ centered at $a_m$ uses a colour of $\mathcal{C}_m$. 
 Moreover, for $a_{m-1}$ there is a set $\mathcal{C}_{m-1}$ of $k$ colours such that $\mathcal{C}_{m} \cap \mathcal{C}_{m-1} = \emptyset$ and such that no star from $\mathcal S$  centered at $a_{m-1}$ uses a colour of $\mathcal{C}_{m-1}$. 
 Continuing in this manner, define sets $\mathcal C_i$ for all $i\leq m$.
 Then, the vertex $(\mathcal{C}_{1}, \dots, \mathcal{C}_{m}) \in B$ is not covered by $\mathcal{S},$ contradicting the fact that $\mathcal{S}$ covers $A \cup B.$

 Using~\eqref{1a1} and~\eqref{ai}, we calculate that
 \begin{align*}
  \sum_{j=1}^m a_j & \geq \sum_{j=1}^{i-1} 1 + \sum_{j=i}^{m} a_i \\
                   & \geq (i-1) + (m-i+1)(r-k(m-i+1)+1)           \\
                   & = r-k+m +  (m-i)(r-2k) -k(m-i)^2             \\
                   & \geq r-k+m,
 \end{align*}
 where the last inequality holds
 since  $\lfloor r/k \rfloor - i \leq r/k - 1$. 
 Thus, $\mathcal S$ contains at least $r - k + m = r - k + \lfloor r/k \rfloor - 1$ stars, which is as desired.
\end{proof}

Observe that the colouring $\varphi$ from Lemma~\ref{lem:lowerbound} attains the bound $r-k+ \lfloor r/k \rfloor -1=r-k+m$ for the size of the cover. 
That is, $A \cup B$ can be covered by $r-k+m$ monochromatic stars: just take $r-k+1$ stars centered at vertex $1 \in A$,  in addition to $m-1$ stars covering the vertices in  $A\setminus\{1 \}$.

\section{Complete graphs}\label{sec:comp_gr}

In this section we prove Theorem \ref{thm:52} and confirm Conjecture~\ref{conj:comp_r-k} for $k \geq r/2 -1$. 
On the road to Theorem~\ref{thm:52}, we prove a result of possible independent interest, Theorem~\ref{mt2}, which bounds the number of vertices in a minimal graph that requires $2$, or $3$, monochromatic components in its cover for some $(r,k)$-colouring.

We say a vertex {\it sees} a colour if it is incident to an edge that carries this colour.

\begin{lemma}\label{lbn}
 Let  $\varphi$ be an $(r,k)$-colouring of $K_n$ such that  $\tc(K_n,\varphi) = t$  and every vertex sees each colour. Then $ k \binom{n}{2} \leq r\left (t-1 + \binom{n-2(t-1)}{2}\right).$
\end{lemma}

\begin{proof}
 Since every edge has $k$ colours it follows that the total number of colours used in $\varphi$, with repetitions allowed, is $k \binom{n}{2}$. 
 On the other hand, since $\tc(K_n,\varphi) = t$, every colour $i$  has at least $t$ components. 
 Each of these components has at least two vertices, by our assumption on $\varphi$. 
 So at most $t-1+ {n - 2(t-1)\choose 2}$  edges have colour $i$ (as in the `worst' case colour $i$ has $t-1$ single-edge components and is complete on the remaining vertices).
\end{proof}

We say that a $r$-colouring $\varphi$ of a graph $G$ is {\it $t$-critical} if $\tc(G,\varphi) = t$ and for each $v\in V(K_n)$, the graph $G\setminus \{v\}$ can be covered by $t-1$ monochromatic components.

\begin{thm}
 \label{mt2}
 Let $\varphi$ be a $t$-critical $(r,k)$-colouring  of $K_n$, for $t\in\{2,3\}$.
 If $t= 2$ then $n \leq r,$ and if $t = 3$ then $n \leq r + \binom{r}{2}.$\\ 
 Moreover, if in $\varphi$ every vertex sees each colour, then $t=3$ and $n \leq \binom{r}{2}$.
\end{thm}

We remark that for the proof of Theorem \ref{thm:52}, we only need Theorem~\ref{mt2} for the special case of colourings $\varphi$ where every vertex sees each colour, and thus $t=3$. 
But as the proof of the whole statement does not require any extra effort, we prefer to state our result as above.

Before we prove Theorem \ref{mt2}, we need some notation.
For a given $t$-critical $r$-colouring $\varphi$ of~$K_n$ we will say that the function $f:V(K_n) \to \cup_{\ell <t} \binom{[r]}{\ell}$ is {\it $t$-critical} for $\varphi$ if $f$ satisfies the following properties:
\begin{enumerate}
 \item[(1)] If $f(v) = \{i_1,\dots,i_\ell\}$ then it is possible to cover all vertices but $v$ by $t-1$ monochromatic components in colours $i_1,\dots,i_\ell.$
 \item[(2)] It holds that $|f(v)| \leq |f'(v)|$ for all functions $f'$ satisfying (1).
\end{enumerate}

Clearly, for every $t$-critical $r$-colouring $\varphi$ of $K_n$ there is a $t$-critical function. 
Moreover, note that for any given vertex $v$, the monochromatic components considered in (1) are non trivial, because otherwise, we can cover the vertices of $K_n$ by $t-1$ monochromatic components, one of which is given by the edge between $v$ and the trivial component.

\begin{proof}[Proof of Theorem \ref{mt2}]
 The first part of Theorem~\ref{mt2} follows from proving injectivity of $t$-critical functions, for $t=2,3$, respectively, since then $n$ is at most the cardinality of the image of injection $f$. 
 The second assertion of the theorem will follow as a by-product of our proof.

 Suppose $u \in V(K_n)$ with $f(u) = \{i\}$. 
 Then, depending on whether $t=2$ or $t=3$, there are one or two monochromatic  components  in colour $i$ covering every vertex other than~$u$. 
 Hence, no edge  incident with $u$ can have colour $i$ (as we need $t$ components to cover $K_n$). 
 Also, every vertex $v \in V(K_n)$ other than $u$ has an incident edge that uses colour $i$. 
 Thus $f(v) \neq \{i\}$ for all $v\neq u.$

 Notice that if every vertex sees each colour, then vertex $u$ from the previous paragraph cannot exist. 
 Thus, in that case, we have $f(u)\neq\{i\}$ for all $u\in V(K_n)$ and all $i\in [r]$. 
 In particular, $t=3$.

 It remains to consider vertices $u \in V(K_n)$ with $f(u) = \{i,j\}$, for $ i \neq j$, and $t=3$. 
 By~(1), there exists monochromatic  components $I_u,J_u$ in colours $i,j$, respectively, covering every vertex of $K_n$ other than $u$. 
 Assume, for the sake of contradiction, that there is a vertex $v\neq u$ with $f(v) = f(u) = \{i,j\},$ and let $I_v,J_v$ the monochromatic components on colours $i,j,$ respectively, covering every vertex of $K_n\setminus \{v\}$. 
 W.l.o.g., we may assume that $v \in I_u$.

 Note that monochromatic components $I_u$ and $I_v$ are vertex-disjoint (as otherwise they would be identical,  but we know that $v \notin I_v$). 
 A second observation is that there is a vertex $w \in J_u \setminus (I_u \cup I_v)$ with $w \neq v$. 
 If not, $I_u \cup I_v$ covers $K_n \setminus \{u\},$ contradicting that the colouring is $3$-critical.

 These observations imply that $J_u = J_v,$ because $w \in J_u \cap J_v$. 
 Hence, $J_u$ covers every vertex of~$K_n$ other than $u$ and $v$. 
 But any monochromatic  component induced by the edge $uv$ covers $u$ and $v$, so $K_n$ is coverable by two monochromatic components, a contradiction. 
 Thus $f(v) \neq \{i,j\}$ for all vertices $v$ other than $u.$
\end{proof}

It is worth noting that the proof of Theorem \ref{mt2} makes no use of the fact that all edges have the same number of colours. 
So the theorem is still valid for a generalised notion of edge-colourings, where each edge is assigned a subset of $[r]$ of arbitrary size.

Now we are ready to prove Theorem \ref{thm:52}.

\begin{proof}[Proof of Theorem \ref{thm:52}]
 By Lemma~\ref{frr-2}, we already know that $\tc_{5,2}(K_n)\geq \tc_{4,2}(K_n)\geq 2$. 
 So we only need to show that $\tc_{5,2}(K_n)\leq 2$.

 For the sake of contradiction, assume $K_n$ has a $(5,2)$-colouring $\varphi$ with $\tc(K_n,\varphi)=3$. 
 We can assume $\varphi$  is $3$-critical. 
 Observe that every triangle is contained in a monochromatic  component, since in every triangle there are  at least two edges sharing  a colour.

 We claim that
 \begin{equation}\label{rarevertices}
  \mbox{each vertex sees each colour.}
 \end{equation}

 For this, assume that $u \in V(K_n)$ does not see colour $5$. 
 Let $U_1,U_2$ be monochromatic components in colours $1,2,$ respectively, both of them containing $u$. 
 Every edge from $u$ to any vertex $v\in V(K_n)\setminus (U_1\cup U_2)$ has the colour set $\{3,4\}$. 
 Such a vertex $v$ must exist, since $\tc(K_n,\varphi) = 3.$ Let $U_3,U_4$ be monochromatic components in colours $3,4,$ respectively, both of them containing $v$. 
 Since $\tc(K_n,\varphi)=3$, there is a vertex $w$ not covered by $U_3\cup U_4$.
 Then $\varphi (uw) = \{1,2\}$. 
 Hence, $vw$ does not have any of the colours $1,2,3$ and $4$, because $v\notin U_1 \cup U_2$ and $w\notin U_3\cup U_4$.
 This contradicts the fact that every edge has two colours, thus proving (\ref{rarevertices}).

 Now, on the one hand, Theorem \ref{mt2} and (\ref{rarevertices}) imply that $n \leq 10$. 
 On the other hand, Lemma \ref{lbn} with $r=5$, $k=2$, $t=3$, together with $(\ref{rarevertices})$, gives that $n > 10$. 
 We thus reached the desired contradiction.
\end{proof}

We conclude this section confirming Conjecture~\ref{conj:comp_r-k} for some special cases, namely, when $k \geq r/2 - 1$.
The proof follows by combining Lemmas~\ref{compr-k} and~\ref{compr-k2} below, and observing that   $\tc_{4,1}(K_n) \leq 3$ (see~\cite{Duc79, Gya77}).

\begin{lem}
	\label{compr-k}
	If $k \geq (r-1)/2$ then $\tc_{r,k}(K_n) \leq r-k.$
\end{lem}

\begin{proof}
	Given an $(r,k)$-coloured  $K_n$, consider the complete bipartite subgraph between any fixed monochromatic  component and the rest of $K_n$. 
	Since this graph inherits an $(r-1,k)$-colouring, and since $r-1 \leq 2k$, Theorem \ref{thm:mainBip} yields a cover by $(r-1)-k+1 = r-k$ monochromatic components.
\end{proof}

\begin{lem}
	\label{compr-k2}
	If $k = r/2 - 1$ and $k\geq 2$, then $\tc_{r,k}(K_n) \leq r-k$.
\end{lem}

\begin{proof}
	Let $A$ be the vertices covered by any fixed monochromatic  component in colour $2k+2,$ and let $(A,B)$ be the complete bipartite graph with partitions $A$ and $B = V(K_n) \setminus A$, with its inherited $(r-1,k)$-colouring.
	We can assume $B\neq\emptyset$.
	
	Fix an edge $vw \in E(A,B)$ with $v \in A$ and $w \in B$, coloured in $\{1,\dots,k\}$, say. 
	Let $A'\subseteq A$ and $B'\subseteq B$ be the sets of vertices not covered by the union of the $k$ monochromatic components in colours $1,\ldots ,k$ that contain the edge $vw$.
	Note that the star centered at $v$ with leaves $B'$  inherits a $(k+1,k)$-colouring and thus, we can cover~$B'$ with two monochromatic stars at $v$. 
	So, since $k+2 = r-k$, we can assume that $A' \neq \emptyset$, and by symmetry, also $B' \neq \emptyset$.
	
	Assume that there is a vertex $w'\in B'$ such that
	\begin{equation} \label{thereisb'}\text{edges $vw'$ and $ww'$ share at least a colour,}\end{equation} say this colour is $k+1$.
	 Then, there are at least $k+1$ monochromatic components, in colours $1,\ldots,k+1$, that contain both $v$ and $w$. 
	 Let $A''\subseteq A'$ and $B''\subseteq B'$ be the sets of vertices not covered by these components.
	
	Observe that every edge between $v$ and $B''$, or between $w$ and $A''$ has colours $\{k+2,\dots,2k+1\}$. 
	So, if there is an edge from $A''$ to $B''$ using one of the colours in $\{k+2,\dots,2k+1\}$, then we can cover all of $A''\cup B''$ with one monochromatic component. 
	Combined with the $k+1$ components from above, we obtain a  cover with $k+2 = r-k$ monochromatic components. 
	So we may assume that every edge between $A''$ and $B''$ avoids colours $\{k+2,\dots,2k+1\}$. 
	In other words, each of these edges has colours $[k],$ and again, we can cover $A \cup B$ with $k+2 = r-k$ monochromatic components.
	
	So from now on, assume that~\eqref{thereisb'} does not hold. 
	Then $k=2$ (and thus, $r=6$).
	For $i \in \{3,4,5\}$ let $B_i := \{w' \in B': i \notin \varphi(vw') \}$. 
	Then $ww'$ is coloured by $\{i,6\}$ if $w' \in B_i$. 
	Hence, it is possible to cover $B'=B_3\cup B_4\cup B_5$ with one monochromatic component in colour $6$. 
	Together with the component $A$, and the $k$ components from above, we obtain  a cover of $A\cup B$ with $k+2 = r-k$ components, as desired.
\end{proof}

\section{Ramsey numbers for $(r,k)$-colourings}\label{sec:ramsey}

In this section, we discuss the set-Ramsey number $\ram_{r,k}(H)$ as defined in the introduction.
We can bound $\ram_{r,k}(H)$ with the help of the usual $r$-colour Ramsey number $\ram_r(H)$. 
In fact, in the same way as we obtained our bounds on $\tc_{r,k}$ in Section~\ref{sec:easy}, one can prove  (see also \cite{XSSL09}) that for every graph $H$ and integers $r > k > 0$,

\begin{equation}\label{r_bound-new}
 \text{$\ram_{r-k+1}(H)\geq \ram_{r,k}(H)\geq \ram_{\lfloor \frac rk\rfloor}(H)$.}
\end{equation}

Both bounds are not best possible as already the example of $r=3$, $k=2$ and $H=K_3$, or $H=K_4$, shows. 
Namely, it is not difficult to show that $\ram_{3,2}(K_3) = 5$, and the value $\ram_{3,2}(K_4) = 10$ follows from results of~\cite{CL78}.
Also for $\ram_{4,2}(K_3)$ the bounds from~\eqref{r_bound-new} are not sharp. 
Corollary~\ref{42K3} near the end of the present section states that $\ram_{4,2}(K_3)\geq 9$, and as we shall see next, this bound is sharp:

\begin{equation}\label{r42}
 \ram_{4,2}(K_3)\leq 9.
\end{equation}

Indeed, in order to see~\eqref{r42}, let  a $(4,2)$-colouring of $K_9$ be given. 
First suppose for some vertex~$v$ there is a colour $i$ appearing on  $5$ edges $vw_1,\ldots,vw_5$. 
If no triple $vw_iw_j$ is an $i$-coloured triangle, then  $w_1,\ldots,w_5$ span a $(3,2)$-colouring, which  has a monochromatic triangle as  $\ram_{3,2}(K_3) = 5$.

So we can assume every vertex is incident with exactly $4$ edges of each colour. 
That is, every colour spans a $4$-regular graph on 9 vertices. We claim each such graph has a triangle. 
Indeed, fixing any edge $uv$, if~$uv$ lies in no triangle, then  $N(u)\cap N(v)=\emptyset$. 
There is a vertex $w\notin N(u)\cup N(v)$, and $w$ has neighbours $u'\in N(u)$, $v'\in N(v)$. 
Since $u',v'$ have degree $4$, either we find a triangle, or we have  $N(u')-w=N(v)-v'$ and $N(v')-w=N(u)-u'$. 
As~$w$ has two more neighbours, we find a triangle. 
This  proves~\eqref{r42}.

\medskip

Furthermore, it is not overly difficult to calculate the values of $r,k,t$ for which $\ram_{r,k}(K_t)$ equals the most trivial bound from below,~$t$.

\begin{equation}\label{r,k,t}
\text{$\ram_{r,k}(K_t)=t$ if and only if $r > (r-k)\tbinom{t}{2}$.}
\end{equation}

For this, observe that each edge misses $r-k$ colours. If $r > (r-k)\tbinom{t}{2}$ holds, then, even if each edge misses disjoint sets of colours, there is still some colour appearing on all edges. So there must be a monochromatic $K_t$. 
On the other hand, if $(r-k)\tbinom{t}{2} \leq r$ we have enough edges to have them miss disjoint sets of colours, and thus $\ram_{r,k}(K_t)>t$.

Observe that in particular, for $t=3$, observation~\eqref{r,k,t} immediately gives that  
$$\ram_{r,k}(K_3)=3\text{ if and only if  }k > 2r/3.$$ So for instance, $\ram_{4,3}(K_3)=3$. 

See Section~\ref{smallSetRamsey} for a summary of small set-Ramsey numbers.

\medskip

Bounds for arbitrary $r$ and $k$ (not necessarily small) can be obtained by density arguments. 
More precisely, if $\frac kr$ surpasses $\frac{t-2}{t-1}$, we can estimate $\ram_{r,k}(K_t)$ using Tur\'an's Theorem:

\begin{prop} \label{boundTuran_f}
 Let $\varepsilon \in (0,1)$, let $t\geq 2$ and let $r > k > 0$. If $\frac{t-2}{t-1} = (1 - \varepsilon)\frac kr$, then $\ram_{r,k}(K_t)\leq \frac1\varepsilon +1$. 
 This bound is sharp if $k=r-1=t-1$ is a prime power, in which case $\ram_{r,k}(K_t) = k^2+1$.
\end{prop}

\begin{proof}
 For the first part, consider any $(r,k)$-colouring of $K_n$ without mono\-chromatic $K_t$. 
 Since every colour has at most $\frac{t-2}{t-1} \cdot \frac{n^2}{2}$ edges, we know that $k \binom{n}{2} \leq r\frac{t-2}{t-1} \cdot \frac{n^2}{2}$ and thus,  $n\leq\frac 1\varepsilon$.

 For the second part, let $\mathcal{P} = (P,\mathcal{L})$ be an affine plane of order $r$ and the complete graph $K = K_{k^2}$ with $V(K) = P$. 
 Colour edge $p_1p_2 \in E(K)$ with $[r]\setminus \{i\}$ if the line containing $p_1,p_2 \in P$ is in the $i$-th parallel class $L_i$ of $\mathcal{L}$. 
 Since for every $i \in [r]$ the $i$-th parallel class $L_i$ consists of $k$ lines, every set of $k+1=r$ points in $P$ contains at least two points that are contained in the same line $l \in L_i$, which proves that the defined colouring contains no monochromatic $K_r$.
\end{proof}

We conclude this section with lower bounds on the set-Ramsey number for odd cycles. 
The next result provides, in particular, the lower bound for~\eqref{r42}.
We remark that for $k$ fixed, and $r$ large enough, the bounds from Theorem~\ref{r_lowerboundC} can be improved, based on recent results from \cite{DJ16} (see Proposition~\ref{p_dj16}).

\begin{thm}
 \label{r_lowerboundC}
 If $\ell$ is odd and $k\geq 2$, then $\ram_{r,k}(C_\ell)> \max \{ 2^{\frac{r-1}{k-1}}, 2^{\lfloor\frac rk\rfloor - 1} (\ell-1) \}$.
\end{thm}

Before we turn to the proof of Theorem~\ref{r_lowerboundC}, let us note that
by using~\eqref{r42}, Theorem~\ref{r_lowerboundC}  has the following immediate corollary.

\begin{cor}\label{42K3}
	We have $\ram_{4,2}(K_3)= 9$.
\end{cor}

\begin{proof}[Proof of Theorem~\ref{r_lowerboundC}]
 To prove $\ram_{r,k}(C_{\ell}) > 2^{\lfloor r/k \rfloor - 1} (\ell-1)$ we use induction on $\lfloor r/k \rfloor$. 
 If  $\lfloor r/k \rfloor=1$, the assertion is trivial, as any $(r,k)$-colouring of $K_{\ell -1}$ will do. 
 For larger values of $\lfloor r/k \rfloor$, it suffices to take two copies of any $(r-k,k)$-coloured $K_{2^{\lfloor r/k \rfloor-2} (\ell-1)}$ without monochromatic $C_\ell$ (such a colouring exists by induction), and give $k$ previously unused colours to every edge between the two copies.

 For the bound $\ram_{r,k}(C_\ell) > n :=2^{\frac{r-1}{k-1}}$, it suffices to find an $(r,k)$-colouring of $K_n$ in which every colour induces a bipartite graph. 
 Such a colouring can be encoded in an $n$-subset $S_{n,r,k}$ of $\{0,1\}^r$ where any two $v,w\in S_{n,r,k}$ differ in at least $k$ entries. 
 (Just consider the complete graph on $S_{n,r,k}$, where we assign colour $i$ to an edge $vw$ if $v$ and $w$ differ at the $i$th entry. If an edge receives more than $k$ colours, just delete some.)

 A set $S_{n,r,k}$ as above clearly exists for $n=2$ and $r=k$, and one can construct a set $S_{2n,r+k-1,k}$ from $S_{n,r,k}$ inductively. 
 Do this by duplicating all members of $S_{n,r,k}$, adding $k-1$ extra entries $0$ to the `original' members, and adding $k-1$ extra entries $1$ to the `clones' (new members). 
 Also, we switch the $r$th entry of each clone: If it was a $0$, we make it a $1$, and if it was a $1$, we make it a $0$.
 Then the new set $S_{2n,r+k-1,k}$ is as desired: every pair of original members and every pair of clones differ in at least~$k$ entries because of the properties of the set $S_{n,r,k}$; every original member differs from its clone in the $k-1$ extra entries and in the switched entry; and finally, every original member differs from all other clones in the $k-1$ extra entries and in at least $k-1\geq 1$ of the original entries (only $k-1$ as one of them might be the one we switched).
\end{proof}

\section{Concluding remarks}\label{sec:conc_rem}

\subsection{Tree covers}
As seen in Section~\ref{sec:easy}, the best lower bound for the tree cover number of complete graphs we know is $\tc_{r,k}(K_n) \geq \lfloor \frac rk \rfloor - 1$, for $n\geq (r-1)^2$ (Lemma~\ref{loboco}). 
On the other hand Conjecture~\ref{conj:comp_r-k} holds for large $n$, although, even if true for all $n$, the conjecture is not tight.
The positive results leave us with the interval $[\lfloor \frac rk \rfloor - 1, r-k]$, if $n$ is large.
We believe that for large  values of $k$ the tree cover number should be closer to the lower bound of this interval.

\begin{problem}
	Determine $\tc_{r,k}(K_n)$ for all $r,k,n$.
\end{problem}

For complete bipartite graphs, for $r<2k$ we do not know more about the true value of $\tc_{r,k}(K_{n,m})$ than the bounds given in Theorems~\ref{thm:mainBip} and~\ref{thm:lowerbound}.

\begin{problem}
	Determine $\tc_{r,k}(K_{n,m})$ for all $r,k,n,m$.
\end{problem}

\subsection{Tree partitions}

In the traditional setting for $r$-coloured complete graphs, Erd\H os, Gy\'arf\'as and Pyber~\cite{EGP91} conjectured a stronger version of 
Conjecture~\ref{conj:Gya77}, namely, they conjectured that a {\it partition} into $r-1$ monochromatic trees should exist.
A weaker version of the latter conjecture, which replaces $r-1$ trees with $r$ trees, was confirmed by Haxell and Kohayakawa~\cite{HK96}, for $n$ sufficiently large compared to $r$.
It would be interesting to explore the tree partition problem for the more general setting of set-colourings. 
Note that the same easy arguments as employed here give that the minimum number of trees needed to partition any $(r,k)$-coloured graph lies in the interval $[\lfloor \frac rk \rfloor - 1, r-k+1]$, if $n$ is large.

One could also study a version this problem for set-colourings of underlying complete multipartite graphs. 
For $k=1$, this problem was addressed by Kaneko, Kano and Suzuki in~\cite{KKS05}.

\subsection{Path/Cycle partitions}\label{cycles}

Another recently very active area involving monochromatic substructures concerns path and cycle covers (see the survey~\cite{Gya16}). 
Let us state the problem here only in a version already adapted to set-colourings of $K_n$. 
The goal is to find the minimum number $\pp(r,k)$ such that in every $(r,k)$-colouring of $K_n$ there are $\pp(r,k)$ disjoint monochromatic paths which together cover all the vertices. The number $\pp(r,k)$ is often called the {\it path partition number}.
We can ask the same question replacing paths with cycles, the respective minimum $\cp(r,k)$ is then called the {\it cycle partition number}. Clearly, $\pp(r,k)\leq\cp(r,k)$, and $\pp(r,r)=\cp(r,r)=1$.

For $k=1$, the following values are known:  $\pp(2,1)=2=\cp(2,1)$, $\pp(3,1) = 3 < \cp(3,1)$, $\pp(4,1) \leq 8$ and it has been conjectured that $\pp(r,1)=r$, while  $\cp(r,1)\geq r+1$, and it is known that $\cp(r,1)$ is bounded from above by a function in $r$ (see~\cite{Gya16}). 

Now, the same trick as used for Lemma~\ref{obsred} (deleting $k-1$ colours from all edges) gives that 
 \begin{equation}\label{ciclos}
\text{$\pp(r,k)\leq \pp(r-k+1,1)$ and $\cp(r,k)\leq \cp(r-k+1,1)$.}
\end{equation}

In particular, these numbers are bounded by functions in $r$.  For  $(r,r-1)$-colourings, we obtain from~\eqref{ciclos} that $$\pp(r,r-1)\leq \cp(r,r-1)\leq \cp(2,1)=2.$$ Hence, $\pp(r,r-1),\cp(r,r-1)\in\{1,2\}$. At first glance, one might think that at least for $\pp(3,2)$, the  answer might be one, and not two, but the following proposition shows that the correct answer is always two.
 
 \begin{prop}
	For $r\geq 2$, we have that $\cp(r,r-1)= \pp(r,r-1)= 2$.
\end{prop}

\begin{proof}
By~\eqref{ciclos}, we only have to show that $\pp(r,r-1)\geq 2$. For this, consider the following construction.

Let $V_1,\dots,V_r$ be pairwise disjoint sets such that $|V_i| > \sum_{j < i} |V_j|+1$, for $i \in \{2,3,\dots,r\}$. We define an $(r,r-1)$-colouring $\varphi$ of $K_n$ on the vertex set $\cup_{i \in [r]} V_i$ as follows: $\varphi(uv) = [r]\setminus \{i\}$ if $u \in V_i, v \in V_j$ and $i < j$; or if $u,v \in V_i$.
Notice that the only edges with colour $i$ and at least one endpoint in $V_i$ are those having their other endpoint in some set $V_j$, with $j < i$. So, since $|\cup_{j<i} V_j|+1 < |V_i|$, no path  of colour $i$ can cover all of $V_i$. 
\end{proof}

\subsection{Set-Ramsey numbers for complete graphs}\label{smallSetRamsey}

In the set-Ramsey numbers setting, let us give a short summary of what is know for $K_3$. For $r=2$, there is nothing interesting to say, since obviously $\ram_{2,2}(K_3)=\ram_{2,1}(K_3)=3$. For $r=3$ it is clear that $\ram_{3,3}(K_3)=3$, it is easy to see that $\ram_{3,2}(K_3)=5$, and it is well-known that $\ram_{3,1}(K_3)=6$. For $r=4$, we have $\ram_{4,4}(K_3)=\ram_{4,3}(K_3)=3$ and  $\ram_{4,2}(K_3)=9$,  as shown in Section~\ref{sec:ramsey}, and $\ram_{4,1}(K_3)$ is the usual 4-coloured Ramsey number for triangles, which is not known.

Therefore, the smallest unknown set-Ramsey number for $K_3$, in terms of $r$ and $k$, is  $\ram_{5,2}(K_3)$. We also do not know $\ram_{5,3}(K_3)$, while $\ram_{5,4}(K_3)=3$ by~\eqref{r,k,t}.
Considering $K_4$, as $\ram_{3,2}(K_4)=10$ by results of~\cite{CL78}, the smallest unknown values correspond to $\ram_{4,2}(K_4)$ and $\ram_{4,3}(K_4)$. 

\subsection{Set-Ramsey numbers for cycles}

The Bondy-Erd\H os conjecture states that $\ram_r(C_\ell) = 2^{r-1}(\ell - 1) + 1$ for every odd $\ell \geq 3$. 
Recently, Jenssen and Skokan~\cite{JS16} proved that the Bondy-Erd\H os conjecture holds for fixed $r$ and sufficiently large odd $n$. 
However, Day and Johnson \cite{DJ16} disprove the Bondy-Erd\H os conjecture by showing that for every odd  $\ell \geq 3$ exist $\varepsilon > 0$ and sufficiently large $r$ such that $\ram_r(C_{\ell}) > (2+\varepsilon)^{r-1}(\ell  - 1)$. 
We can imitate their construction to see that, analogously:

\begin{prop}\label{p_dj16}
	For all $k \geq 2$ and odd $\ell$ there are  $\varepsilon = \varepsilon(\ell) > 0$ and $f = f(\ell) > 1$ such that for sufficiently large $r$ we have $\ram_{r,k}(C_{\ell}) > (2+\varepsilon)^{r-f}(\ell - 1)$.
\end{prop}

We include a sketch of the proof of Proposition~\ref{p_dj16} for readers who are familiar with the construction of Day and Johnson in~\cite{DJ16}.

\begin{proof}[Sketch of a proof for Proposition~\ref{p_dj16}]
	Let $k \geq 2$ be fixed. As in \cite{DJ16} one can show the existence of $(r,k)$-colourings of $K_{2^r+1}$ with arbitrarily long odd girth. Let $f'$ be the smallest integer such that there is an $(f',k)$-colouring $\varphi_1$ of $K_{2^{f'}+1}$ with odd girth strictly greater than $\ell$. Let $\varphi_2$ be the $(c+1,k)$-colouring of the complete graph on $2^{\lfloor c/k \rfloor}(\ell-1)$ vertices avoiding a monochromatic $C_{\ell}$, as given by Theorem \ref{r_lowerboundC}. If $r = mf' + c$, with $f' > c \geq 0$, then, following the construction in \cite{DJ16}, one can define an $(r,k)$-colouring of the complete graph on
	$(2^{f'} + 1)^m \cdot 2^{\lfloor c/k \rfloor} (\ell - 1) > (2+\varepsilon)^{r-f}$ vertices.
\end{proof}

\section*{Acknowledgments}

We would like to thank two anonymous referees for their careful reading and helpful suggestions.

Both authors acknowledge support by Millenium Nucleus Information and Coordination in Networks ICM/FIC RC130003.
The first author also was supported by CONICYT Doctoral Fellowship 21141116.
The second author also received support by Fondecyt Regular grant 1140766 and CMM-Basal.


\begin{thebibliography}{10}
	
	\bibitem{AEGM02}
	{\sc Alon, N., Erd{\H{o}}s, P., Gunderson, D.~S., and Molloy, M.}
	\newblock A {R}amsey-type problem and the tur{\'a}n numbers.
	\newblock {\em Journal of Graph Theory 40}, 2 (2002), 120--129.
	
	\bibitem{BGS11}
	{\sc Balister, P.~N., Gy{\H{o}}ri, E., and Schelp, R.~H.}
	\newblock Coloring vertices and edges of a graph by nonempty subsets of a set.
	\newblock {\em European Journal of Combinatorics 32}, 4 (2011), 533--537.
	
	\bibitem{BT79}
	{\sc Bollob{\'a}s, B., and Thomason, A.}
	\newblock Set colourings of graphs.
	\newblock {\em Discrete Mathematics 25}, 1 (1979), 21--26.
	
	\bibitem{CFGLT12}
	{\sc Chen, G., Fujita, S., Gy{\'a}rf{\'a}s, A., Lehel, J., and T{\'o}th, A.}
	\newblock Around a biclique cover conjecture.
	\newblock {\em arXiv preprint arXiv:1212.6861\/} (2012).
	
	\bibitem{CL78}
	{\sc Chung, K., and Liu, C.~L.}
	\newblock A generalization of {R}amsey theory for graphs.
	\newblock {\em Discrete Mathematics 21}, 2 (1978), 117--127.
	
	\bibitem{DJ16}
	{\sc Day, A.~N., and Johnson, J.~R.}
	\newblock Multicolour {R}amsey numbers of odd cycles.
	\newblock {\em arXiv preprint arXiv:1602.07607\/} (2016).
	
	\bibitem{Duc79}
	{\sc Duchet, P.}
	\newblock {\em Repr{\'e}sentations, noyaux en th{\'e}orie des graphes et
		hypergraphes}.
	\newblock PhD thesis, Th{\'e}se, Paris, 1979.
	
	\bibitem{EGP91}
	{\sc Erd{\H{o}}s, P., Gy{\'a}rf{\'a}s, A., and Pyber, L.}
	\newblock Vertex coverings by monochromatic cycles and trees.
	\newblock {\em Journal of Combinatorial Theory, Series B 51}, 1 (1991), 90--95.
	
	\bibitem{EHR65}
	{\sc Erd{\H{o}}s, P., Hajnal, A., and Rado, R.}
	\newblock Partition relations for cardinal numbers.
	\newblock {\em Acta Mathematica Hungarica 16}, 1-2 (1965), 93--196.
	
	\bibitem{Gya77}
	{\sc Gy{\'a}rf{\'a}s, A.}
	\newblock Partition coverings and blocking sets in hypergraphs.
	\newblock {\em Communications of the Computer and Automation Institute of the
		Hungarian Academy of Sciences 71\/} (1977), 62.
	
	\bibitem{Gya16}
	{\sc Gy{\'a}rf{\'a}s, A.}
	\newblock Vertex covers by monochromatic pieces -- a survey of results and
	problems.
	\newblock {\em Discrete Mathematics 339\/} (2016), 1970--1977.
	
	\bibitem{HM99}
	{\sc Harborth, H., and M{\"o}ller, M.}
	\newblock Weakened {R}amsey numbers.
	\newblock {\em Discrete Applied Mathematics 95}, 1-3 (1999), 279--284.
	
	\bibitem{HK96}
	{\sc Haxell, P.~E., and Kohayakawa, Y.}
	\newblock Partitioning by monochromatic trees.
	\newblock {\em Journal of Combinatorial Theory, Series B 68}, 2 (1996),
	218--222.
	
	\bibitem{Hed09}
	{\sc Hegde, S.}
	\newblock Set colorings of graphs.
	\newblock {\em European Journal of Combinatorics 30}, 4 (2009), 986--995.
	
	\bibitem{Hen70}
	{\sc Henderson, J.~R.}
	\newblock {\em Permutation decompositions of $(0, 1)$-matrices and
		decomposition transversals}.
	\newblock PhD thesis, California Institute of Technology, 1970.
	
	\bibitem{JS16}
	{\sc Jenssen, M., and Skokan, J.}
	\newblock Exact {R}amsey numbers of odd cycles via nonlinear optimisation.
	\newblock {\em arXiv preprint arXiv:1608.05705\/} (2016).
	
	\bibitem{KKS05}
	{\sc Kaneko, A., Kano, M., and Suzuki, K.}
	\newblock Partitioning complete multipartite graphs by monochromatic trees.
	\newblock {\em Journal of Graph Theory 48}, 2 (2005), 133--141.
	
	\bibitem{Kir11}
	{\sc Kir{\'a}ly, Z.}
	\newblock Monochromatic components in edge-colored complete uniform
	hypergraphs.
	\newblock {\em Electronic Notes in Discrete Mathematics 38\/} (2011), 517--521.
	
	\bibitem{KT17}
	{\sc Kir{\'a}ly, Z., and T\'othm\'er\'esz, L.}
	\newblock On {R}yser's conjecture: $ t $-intersecting and degree-bounded
	hypergraphs, covering by heterogeneous sets.
	\newblock {\em arXiv preprint arXiv:1705.10024\/} (2017).
	
	\bibitem{XSSL09}
	{\sc Xu, X., Shao, Z., Su, W., and Li, Z.}
	\newblock Set-coloring of edges and multigraph {R}amsey numbers.
	\newblock {\em Graphs and Combinatorics 25}, 6 (2009), 863--870.
	
\end{thebibliography}
\bibliographystyle{acm}

\end{document}